\documentclass[12pt]{amsart}
\numberwithin{equation}{section}
\usepackage[dvips]{graphicx}
\usepackage{amssymb}
\usepackage{amsmath}
\usepackage{latexsym}
\newcommand{\eps}{\varepsilon}
\def\R{\mathbb R}
\def\C{\mathbb C}
\def\Z{\mathbb Z}

\def\N{\mathbb N}
\def\im{\operatorname{Im}}
\def\re{\operatorname{Re}}
\def\dim{\operatorname{dim}}

\def\dist{\operatorname{dist}}

\def\length{\operatorname{length}}
\def\diam{\operatorname{diam}}
\def\area{\operatorname{area}}

\newtheorem{la}{Lemma}[section]
\newtheorem{thm}{Theorem}[section]

\theoremstyle{definition}
\newtheorem{defin}{Definition}[section]
\theoremstyle{remark}
\newtheorem*{rem}{Remark}

\title[Hausdorff dimension of Julia sets]{The growth rate of an entire function and the Hausdorff
dimension of its Julia set} 
\thanks{All three authors were supported by the EU Research Training
Network CODY. The first author was also supported by the G.I.F.,
the German--Israeli Foundation for Scientific Research and
Development, Grant G-809-234.6/2003 and the ESF Research
Networking Programme HCAA. The second author was also supported
by Polish MNiSW Grant N N201 0234 33 and PW Grant 504G 1120 0011
000. The latter grant supported a visit of the first and third
authors to Warsaw, during which most of the work for this paper
was carried out}
\author{Walter Bergweiler}
\address{Mathematisches Seminar,
Christian--Albrechts--Universit\"at zu Kiel,
Lude\-wig--Meyn--Str.~4,
D--24098 Kiel,
Germany}
\email{bergweiler@math.uni-kiel.de}
\author{Bogus\l awa Karpi\'nska}
\address{Faculty of Mathematics and Information Science,
Warsaw University of Technology,
Pl.\ Politechniki 1, 00-661 Warszawa, Poland}
\email{bkarpin@mini.pw.edu.pl}
\author{Gwyneth M.\ Stallard }
\address{Department of Mathematics and Statistics, The Open University, Walton Hall,
Milton Keynes MK7 6AA, United Kingdom}
\email{g.m.stallard@open.ac.uk}
\begin{document}
\begin{abstract}
Let $f$ be a transcendental entire function in the
Eremenko-Lyubich class~$B$. We give a lower bound for the
Hausdorff dimension of the Julia set of $f$ that depends on the
growth of~$f$. This estimate is best possible and is obtained by
proving a more general result concerning the size of the escaping
set of a function with a logarithmic tract.
\end{abstract}
\maketitle
\section{Introduction and main result}

Let $f$ be a transcendental entire function and denote by $f^{n},
n \in \N$, the $n$th iterate of~$f$. The {\it Fatou set}, $F(f)$,
is defined to be the set of points, $z \in \C$, such that
$(f^{n})_{n \in \N}$ forms a normal family in some neighbourhood
of~$z$. The complement, $J(f)$, of $F(f)$ is called the {\it
Julia set} of~$f$. An introduction to the basic properties of
these sets can be found in, for example,~\cite{Ber93}.

The Hausdorff dimension of the Julia set of an entire function
$f$ was first considered by McMullen~\cite{McM} who proved that
$\dim J(f) = 2$ if $f(z) = \lambda e^z$, where $\lambda \in \C
\setminus \{0\}$. Taniguchi~\cite{Tan03} extended this result to
functions $f$ of the form
\begin{equation}\label{aa}
f(z) = \int_0^zP(t) e^{Q(t)} dt + c,
\end{equation}
where $P$ and $Q$ are polynomials and $c \in \C$.

There is a close relationship between the Julia set and the {\it
escaping set}
\[
 I(f) = \{z: f^n(z) \to \infty \mbox{ as } n \to \infty \}
\]
which was first studied for a general transcendental entire
function $f$ by Eremenko~\cite{Ere89}. Among other results,
Eremenko proved that the Julia set is always equal to the boundary
of the escaping set.

An important role in complex dynamics is played by the
Eremenko-Lyubich class $B$ consisting of all transcendental
entire functions for which the set of critical values and finite
asymptotic values is bounded. This class contains the functions
considered by McMullen and Taniguchi mentioned above. Eremenko
and Lyubich~\cite{EL} proved that, if $f \in B$, then $I(f)
\subset J(f)$. Thus, for such functions, a lower bound for the
size of the Julia set can be obtained by estimating the size of
the escaping set. An alternative method for obtaining a lower
bound for the size of the Julia set of a function in the class
$B$ is given in~\cite{BKZ}.

The goal of this paper is to relate the Hausdorff dimension of
the Julia set of a function in the class $B$ to the growth rate of
the function. Recall that the {\it order} of an entire function
$f$ is defined by
\[
 \rho(f) = \limsup _{r \to \infty} \frac{\log \log M(r,f)}{\log
 r},
\]
where $M(r,f) = \max_{|z|=r} |f(z)|$. An entire function $f$ has
finite order if and only if there exists $\rho(f) \in [0,\infty)$
such that, for each $\eps
> 0$, there exists $r_{\eps}>0$ such that
\[
 |f(z)| \leq \exp\left( |z|^{\rho(f) + \eps} \right) \; \mbox{
 for } |z| > r_{\eps}.
\]
Note that, if $f$ is of the form \eqref{aa}, then the order of $f$
is equal to the degree of~$Q$.

Bara\'nski~\cite{Bar} and Schubert~\cite{Sch07} proved that $\dim
J(f) = 2$ for any function $f$ of finite order in the class~$B$.
The hypothesis that $f$ has finite order cannot be omitted since
it is known~\cite{Sta91} that, for each $\eps>0$, there exists a
function $f \in B$ for which $\dim J(f) < 1 + \eps$. In fact, for
each $d \in (1,2)$ there exists a function $f \in B$ for which
$\dim J(f) = d$; see~\cite{Sta00}. On the other hand, it was
shown in~\cite{Sta96} that $\dim J(f) > 1$ for any $f \in B$.

We now state the main result of the paper. We note that the
examples in~\cite{Sta00} show that this estimate of $\dim J(f)$
is best possible.

\begin{thm} \label{thm1}
Let $f$ be an entire function in the class $B$ and let $q\geq 1$.
Suppose that, for each $\eps>0$, there exists $r_\eps>0$ such that
\begin{equation}\label{a}
|f(z)|\leq \exp\left( \exp\left(\left(\log|z|\right)^{q+\eps}\right)\right)
\quad \text{for } |z|\geq r_\eps.
\end{equation}
Then
\[
\dim J(f)\geq 1+\frac1q.
\]
\end{thm}

Our result shows that $\dim J(f)=2$ for $f\in B$ not only if $f$
has finite order, but more generally if
\begin{equation}\label{Q}
\limsup_{r\to\infty} \frac{\log\log\log M(r,f)}{\log \log r}=1.
\end{equation}
It thus strengthens the results of Bara\'nski~\cite{Bar} and
Schubert~\cite{Sch07}. Examples of functions in the class $B$ can
be constructed, for example, by using contour integrals;
see~\cite{PS}, \cite{Sta91} and \cite{Sta00}. This technique
yields examples of functions $f \in B$ which have infinite order
but satisfy \eqref{Q}. For any $q>1$, examples of functions $f \in
B$ which have infinite order and satisfy \eqref{a} were
constructed in \cite{RRRS}. These examples have the additional
property that all the path-connected components of $J(f)$ are
points.

Next we note that, if $f\in B$, then $\rho(f) \geq \frac12$. This
observation seems to have appeared first in~\cite{L},~\cite{BE};
see also~\cite[Lemma~3.5]{RS05c}. This implies that a function
$f\in B$ cannot satisfy \eqref{a} for some $q<1$.

Finally we note that the hypothesis \eqref{a} can also be written
in the form
\[
 \limsup_{r \to \infty} \frac{\log \log \log M(r,f)}{\log \log
 r}\leq
 q.
\]
Such limits or, more generally, limits of the form
\[
 \limsup_{r \to \infty} \frac{\log^{i+k}M(r,f)}{\log^kr},
\]
for certain $i,k \geq 0$, have been considered, for example,
in~\cite[Chapter IV]{JV}.

The main tool used by Eremenko and Lyubich to study a function $f$
in the class $B$ was a {\it logarithmic change of variable}.
Choose $R>|f(0)|$ such that $\Delta_R= \{z \in \C: |z| > R\}$
contains no critical values and no asymptotic values of~$f$. Then
every component $D$ of $f^{-1}(\Delta_R)$ is simply connected and
$f:D \to \Delta_R$ is a universal covering. Now let
$H=\{z\in\C:\re z> \log R\}$. The map $\exp: H\to \Delta_R$ is
also a universal covering and so there exists a biholomorphic map
$G:D\to H$ such that $f=\exp\circ G$. We define
$F:\exp^{-1}(D)\to H$ by $F(z)=G(e^z)$ so that $\exp F(z) =
f(e^z)$. We say that $F$ is the function obtained from $f$ by a
logarithmic change of variable.

In many applications of this method it is irrelevant how $f$
behaves outside~$D$, or whether $f$ is even defined outside~$D$.
This leads to the following definition.

\begin{defin} \label{defn1}
Let $D \subset \C$ be an unbounded domain in $\C$ whose boundary
consists of piecewise smooth curves. Suppose that the complement
of $D$ is unbounded. Let $f$ be a complex-valued function whose
domain of definition contains the closure $\overline{D}$ of~$D$.
Then $D$ is called a {\it logarithmic tract} of $f$ if $f$ is
holomorphic in $D$ and continuous in $\overline{D}$ and if there
exists $R>0$ such that $f:D \to \Delta_R$ is a universal covering.
\end{defin}

If $D$ is a logarithmic tract of $f$, then
\[
I(f,D) = \{z \in D: f^n(z) \in D \mbox{ for all } n \in \N \mbox{
and } f^n(z) \to \infty \mbox{ as } n \to \infty\}.
\]
In fact it was shown in~\cite[Theorem 2.4]{Rem2} that $I(f,D)$
always has at least one unbounded component; see also~\cite{BRS}
for a generalisation to the case when $D$ is a direct tract.

Note that, if $f \in B$, then $f$ has a logarithmic tract~$D$.
Clearly $I(f,D) \subset I(f)$ and so $I(f,D) \subset J(f)$ by the
result of Eremenko and Lyubich~\cite{EL} mentioned earlier. Hence
Theorem~\ref{thm1} follows from the following general result.

\begin{thm} \label{thm2}
Let $f$ be a function with a logarithmic tract~$D$. Suppose that,
for each $\eps>0$, there exists $r_\eps>0$ such that \eqref{a}
holds for $z\in D$. Then $\dim I(f,D)\geq 1+1/q$.
\end{thm}

In order to prove Theorem~\ref{thm2}, we work with the function
$F: \exp^{-1}(D) \to H$ obtained from $f$ by a logarithmic change
of variable. For each $p>q-1$, we construct a set $E_p \subset
\exp^{-1}(D)$ such that
\[
  \re F^n(z) \to \infty \; \mbox{ for } z \in E_p
\]
and
\[
  \dim E_p \geq 1 + \frac{1}{1+p}.
\]
 We note that $\exp(E_p) \subset I(f,D)$. Since $E_p$ and $\exp (E_p)$ have
the same Hausdorff dimension and $p$ can be chosen to be
arbitrarily close to $q-1$, this is sufficient to prove
Theorem~\ref{thm2}.\\

To construct the sets $E_p$, we use a generalisation of the method
used by McMullen in~\cite{McM}. As in~\cite{McM}, the sets $\exp
E_p$ that we construct consist of points that `zip to infinity';
that is, they belong to the set
\[
 Z(f,D) = \left\{ z \in I(f,D): \frac{1}{n} \log \log |f^n(z)| \to \infty
 \mbox{ as } n \to \infty \right\}.
\]
In our situation, however, more sophisticated arguments are needed
to construct such points. The machinery required for this
construction is set up in Section 3 where we introduce the notion
of an `admissible square' -- a square where $F$ grows regularly in a certain sense.
The sets $E_p$ consist of points whose forward iterates
all lie in an admissible square. We estimate the dimensions of the
sets $E_p$ in Section 4. Our calculations are based on ideas
similar to those used by McMullen in~\cite{McM} -- again, more
delicate arguments are needed as we do not have uniform bounds on
the quantities involved.

We note that in contrast to
Bara\'nski~\cite{Bar} and Schubert~\cite{Sch07} we do not use Ahlfors' distortion theorem.

\section{Preliminary lemmas}
We begin with the following lemma about real functions. It is
very similar to~\cite[Lemma~3]{Ber90}, but we include the proof
for completeness.
\begin{la} \label{la1}
Let $\alpha, \beta:[c,\infty)\to \R$ be continuous and increasing.
Suppose that $\beta$ is differentiable, that
$\alpha$ is absolutely continuous, that
$\alpha(x)\leq \beta(x)$, that $\lim_{x\to\infty}\beta(x)=\infty$
and that $\beta'(x)>0$. Define $\psi: [\beta(c),\infty)\to
(0,\infty)$ by
$\psi(t)=\beta'\left(\beta^{-1}(t)\right)=1/(\beta^{-1})'(t)$.
If $K>1$, then
\begin{equation} \label{la1a}
\alpha'(x)\leq K\psi(\alpha(x))
\end{equation}
on a set of $x$-values of lower density at least $(K-1)/K$.
\end{la}
Of course, the inequality \eqref{la1a} makes sense only for values of
$x$ where $\alpha$ is differentiable, but
absolutely continuous functions are differentiable almost everywhere. Thus, if $L_K$ denotes
the set where \eqref{la1a} holds, then the points where $\alpha$ is not differentiable are in the complement of $L_K$.
\begin{proof}
For $y>c$ we define
\[
C_y=\{x \in [c,y] : \alpha'(x)>K\psi(\alpha(x))\}.
\]
Then
\begin{eqnarray*}
K\int_{C_y}dx&\leq&\int_{C_y}\frac{\alpha'(x)}{\psi(\alpha(x))}dx\\
&\leq&\int_{c}^y\frac{\alpha'(x)}{\psi(\alpha(x))}dx\\
&=&\int_{\alpha(c)}^{\alpha(y)}\frac{du}{\psi(u)}\\
&=&\beta^{-1}(\alpha(y))-\beta^{-1}(\alpha(c))\\
&\leq&y-\beta^{-1}(\alpha(c)),
\end{eqnarray*}
and we deduce that the set of $x$-values where $\alpha'(x)>K\psi(\alpha(x))$
has upper density at most $1/K$. The conclusion follows.
\end{proof}
We next recall the following classical result. Inequalities
\eqref{la2a} and \eqref{la2b} are Koebe's distortion theorem and
\eqref{la2c} is Koebe's $\tfrac{1}{4}$-theorem. Here, and
throughout the paper, $B(a,r)$ denotes the open disk around $a$ of
radius~$r$.
\begin{la} \label{la2}
Let $g:B(a,r)\to\C$ be univalent, $\rho \in (0,1)$ and
$z\in B(a,\rho r)$. Then
\begin{equation}\label{la2a}
\frac{\rho}{(1+\rho)^2}
|g'(a)|r
\leq |g(z)-g(a)|
\leq
\frac{\rho}{(1-\rho)^2}
|g'(a)|r ,
\end{equation}
\begin{equation}\label{la2b}
\frac{1-\rho}{(1+\rho)^3}
|g'(a)|
\leq |g'(z)|
\leq
\frac{1+\rho}{(1-\rho)^3}
|g'(a)|
\end{equation}
and
\begin{equation}\label{la2c}
g(B(a,r))\supset
B\left(g(a),\tfrac14 |g'(a)|r\right).
\end{equation}
\end{la}

\begin{rem}
If $\rho = \frac12$ then \eqref{la2b} takes the form
\[
 \tfrac{4}{27} |g'(a)| \leq |g'(z)| \leq 12 |g'(a)|
\]
and so, if $z,w \in B(a,\frac{1}{2}r)$,
\[
 \frac{1}{81} \leq \left| \frac{g'(z)}{g'(w)} \right| \leq 81.
\]
\end{rem}

The following result is a
simple consequence of Koebe's distortion theorem.
\begin{la} \label{la3}
Let $g:B(a,r)\to\C$ be univalent, $\rho\in\left(0,\frac12\right)$
and $z,w\in B(a,\rho r)$. Then
\[
|g(z)-g(w)-g'(a)(z-w)|\leq 26  |g'(a)| \rho |z-w|.
\]
\end{la}
\begin{proof}
It follows from~\eqref{la2b} that if $\zeta\in B\left(a,\frac12 r\right)$,
then
\[
|g'(\zeta)-g'(a)|
\leq |g'(\zeta)|+|g'(a)|\leq 13 |g'(a)|.
\]
Schwarz's lemma yields
\[
|g'(\zeta)-g'(a)|
\leq 26 |g'(a)| \frac{|\zeta - a|}{r}
\]
for $\zeta\in B\left(a,\frac12 r\right)$.
Hence
\[
|g(z)-g(w)-g'(a)(z-w)| =\left|\int_w^z \left(g'(\zeta)
-g'(a)\right) d\zeta \right| \leq 26  |g'(a)| \rho |z-w|,
 \]
for $z,w\in B(a,\rho r)$.
\end{proof}

We shall also need the following version of Vitali's
lemma~\cite[Lemma 4.8]{Fal03}.
\begin{la} \label{la4}
Let $\left\{B(x_i,r_i):i\in I\right\}$ be a collection of
balls in $\R^n$ whose union is bounded. Then there exists
a finite subset $E$ of $I$ such that
$B(x_i,r_i)\cap B(x_j,r_j)=\emptyset$ for $i,j\in E$, $i\neq j$,
and
\[
\bigcup_{i\in I} B(x_i,r_i)
\subset
\bigcup_{i\in E} B(x_i,4 r_i).
\]
\end{la}

\section{Admissible squares}
Let $f$  and $D$ be as in Theorem~\ref{thm2}.
We may assume
that $R=1$ in the definition
of the logarithmic tract and that $0\notin D$.
Let $H=\{z\in\C:\re z>0\}$ be the right half-plane and let
$F: \exp^{-1}(D)\to H$
be the function obtained from $f$ by a logarithmic
change of variable, as described in Section 1.
Note that $F$ is $2\pi i$-periodic and the restriction of $F$ to
a component of $\exp^{-1}(D)$ maps this component
bijectively onto~$H$.

We now fix $\eps>0$ and $p>q-1 + \eps$. Recall that we are aiming
to construct a set $E_p \subset \exp^{-1}(D)$ such that
\[
  \re F^n(z) \to \infty \; \mbox{ for } z \in E_p.
\]
In order to do this, we
let $x_0=\inf\{\re z: z\in \exp^{-1}(D)\}$ and consider the function
$h:(x_0,\infty)\to (0,\infty)$ defined by
\[
h(x)=\max_{y\in\R} \re F(x+iy).
\]
Note that $h$ is increasing by the maximum principle. Moreover,
$h$ is convex by analogy to Hadamard's three circles theorem.
Thus $h$ has left and right derivatives at all points. If
$z_x=x+i y_x$ is a point such that $h(x)=\re F(z_x)$, then
$F'(z_x)$ is real and lies between the left and right derivative
of $h$ at~$x$. Except for the countable set $C$ where $h$ is not
differentiable, we thus have
\begin{equation}\label{a1}
h'(x)=F'(z_x).
\end{equation}

We now obtain estimates for the size of $h$ and $h'$.

\begin{la} \label{la4a}
Let $h:(x_0,\infty)\to (0,\infty)$ and the countable set $C$ be
defined as above. Then there exists $x_{\eps} \geq x_0$ and a set
$L \subset (x_0, \infty) \setminus C$ of density $1$ such that
\begin{equation}\label{a2}
h(x)\leq \exp\left(x^{q+\eps}\right)\quad \text{for } x \in
(x_\eps, \infty),
\end{equation}

\begin{equation}\label{c}
\frac{h'(x)}{h(x)}\leq x^p\quad \text{for } x\in L,
\end{equation}

\begin{equation}\label{c1}
\frac{h'(x)}{h(x)}\geq \frac{1}{4\pi} \quad \text{for } x\in
(x_0,\infty) \setminus C,
\end{equation}

\begin{equation}\label{c2}
h(x) \geq \exp\left(\tfrac{1}{13}x\right) \quad \text{for } x\in
(x_\eps,\infty),
\end{equation}
and
\begin{equation}\label{c3}
h'(x) \geq \exp\left(\tfrac{1}{14}x\right) \quad \text{for } x\in
(x_\eps, \infty) \setminus C.
\end{equation}

\end{la}

\begin{proof}
The upper bound \eqref{a2} for $h$ follows directly from
hypothesis \eqref{a}.

To obtain an estimate for $h'$ we note that it follows
from~\eqref{a2} that we can apply Lemma~\ref{la1} with
$\alpha(x)=h(x)$ and $\beta(x)=\exp\left(x^{q+\eps}\right)$. We
have $\beta^{-1}(t)=(\log t)^{1/(q+\eps)}$, $\beta'(x)=(q+\eps)
\beta(x)x^{q+\eps-1}$ and
\[
\psi(t)=(q+\eps) t \left(\log t\right)^{\frac{q+\eps-1}{q+\eps}}
\]
so that
\begin{equation}\label{b}
h'(x)\leq K (q+\eps) h(x)\left(\log h(x)\right)^{\frac{q+\eps-1}{q+\eps}}
\leq  K (q+\eps) h(x) x^{q+\eps-1}
\end{equation}
on a set of lower density at least $(K-1)/K$. Since $p>q-1+
\eps$, the right hand side of~\eqref{b} is smaller than $h(x)x^p$
for large~$x$, if $K>1$ is fixed. The upper bound \eqref{c} for
$h'/h$ now follows.

Now recall that  $z_x=x+i y_x$ is a  point such that $h(x)=\re
F(z_x)$. It follows from Koebe's $\tfrac{1}{4}$-theorem
\eqref{la2c} and from \eqref{a1} that if $\varphi$ is the branch
of $F^{-1}$ that maps $F(z_x)$ to $z_x$, then $\varphi(B(F(z_x),
h(x))$ contains a disk around $z_x$ of radius $r$, where
\[
r=\frac{h(x)\varphi'(F(z_x))}{4}= \frac{h(x)}{4F'(z_x)} =
\frac{h(x)}{4h'(x)} \quad \text {for } x \in (x_0, \infty)
\setminus C.
\]
 On the other hand, $\varphi(B(F(z_x), h(x))\subset
\exp^{-1}(D)$ and $\exp^{-1}(D)$ does not contain disks of radius
greater than~$\pi$. The lower bound \eqref{c1} for $h'/h$ now
follows. Integrating \eqref{c1} and noting that $4\pi<13$, we
obtain \eqref{c2}. The lower bound \eqref{c3} for $h'$ follows
from \eqref{c1} and \eqref{c2}.
\end{proof}

We are now in a position to define the key idea of an admissible
square.

\begin{defin} \label{defn2}
For $z\in\C$ and $r>0$ we consider the square
\[
S(z,r)=\left\{ \zeta \in\C: |\re \zeta -\re z|\leq r,
|\im  \zeta -\im z|\leq r\right\}.
\]
We call $z$ the {\em centre} of $S(z,r)$. We say that $S(z,r)$ is
{\em admissible} if $100<r<\frac12\re z$ and
\[
\length([\re z-r,\re z+r]\cap L)\geq \tfrac74 r,
\]
where $\length(\cdot)$ denotes the one-dimensional Lebesgue
measure and $L$ is the set of density~$1$ from~\eqref{c}.
\end{defin}

The following result is the main tool that we use in the
construction of the set $E_p$.

\begin{la} \label{la5}
Given $\tau>1$, there exist positive constants $c_0,c_1,c_2,c_3$ with
the following properties:

If $S(z,r)$ is an admissible square and $x=\re z> c_0$, then
there exist $m\in\N$ with $m>c_3 r x^p$, compact subsets
$A_1,A_2,\dots,A_m$ of $S(z,\frac14 r)$ and points
$a_1,a_2,\dots,a_m$ in $S(z, \frac14 r)$ such that $F$ maps $A_j$
bijectively onto an admissible square centred at $F(a_j)$,
\begin{equation}\label{A}
B\left(a_j,\frac{c_1}{x^p}\right) \subset A_j \subset
B\left(a_j,\frac{c_2}{x^p}\right) \subset B\left(a_j,1\right)
\subset S\left(z, \tfrac14 r\right),
\end{equation}
\begin{equation}\label{B}
B\left(a_j,\frac{\tau c_2}{x^p}\right) \subset \left\{ \zeta
\in\C: |\re \zeta -\re z|\leq \tfrac14 r, |\im  \zeta -\im z|\leq
\pi+1 \right\}
\end{equation}
and
\begin{equation}\label{C}
\re F(a_j)\geq \exp\left(\tfrac{1}{15}x\right)\geq c_0
\end{equation}
for $j=1,2,\dots, m$. Moreover,
\begin{equation}\label{D}
\re a_{j+1} > \re a_{j} + \frac{\tau c_2}{x^p}
\end{equation}
for $j=1,2,\dots, m-1$.
\end{la}
\begin{proof}
Let
\begin{equation}\label{d4}
L'=\left[x-\tfrac14 r+1,x+\tfrac14 r-1\right]\cap L.
\end{equation}
Since $S(z,r)$ is admissible we have
\[
\length(L')\geq \tfrac14 r -2 \geq \tfrac15 r.
\]
We apply Lemma~\ref{la4} to the intervals
\[
\left( u-3\frac{h(u)}{h'(u)},u+3\frac{h(u)}{h'(u)}\right),\quad u\in L'.
\]
We obtain $u_1,u_2,\dots,u_n\in L'$ such that if we put
\begin{equation} \label{d5}
r_k=\frac{h(u_k)}{h'(u_k)},
\end{equation}
then
\begin{equation}\label{d3}
(u_j-3r_j,u_j+3r_j) \cap (u_k-3r_k,u_k+3r_k) = \emptyset \quad
\text{ for } j \neq k,
\end{equation}
and
\[
L'\subset \bigcup_{k=1}^n (u_k-12 r_k,u_k+12 r_k).
\]
It follows that
\begin{equation}\label{d}
\sum_{k=1}^n r_k \geq \tfrac{1}{24}\length(L') \geq
\tfrac{1}{120}r.
\end{equation}
Now let $w_k=z_{u_k}$; that is, $\re w_k=u_k$ and $h(u_k)=\re
F(w_k)$. Since $F$ is $2\pi i$-periodic, we may choose $w_k$ such
that
\begin{equation}\label{d2}
|\im w_k - \im z|\leq \pi.
\end{equation} We
denote by $\varphi_k$ the branch of the inverse function of $F$
for which $\varphi_k(F(w_k))=w_k$. Then
\begin{equation}\label{d1}
\varphi_k'(F(w_k))=\frac{1}{F'(w_k)}=\frac{1}{h'(u_k)}
\end{equation}
by \eqref{a1} so that
\begin{equation}\label{e}
h(u_k)\varphi_k'(F(w_k))=r_k.
\end{equation}
Let
\[
W_k=\varphi_k\left( S\left(F(w_k),\tfrac14 h(u_k)\right)\right).
\]
Since $\varphi_k$ is univalent in the right half-plane $H$ and
\[
B\left(F(w_k),\tfrac14 h(u_k)\right) \subset
S\left(F(w_k),\tfrac14 h(u_k)\right) \subset
B\left(F(w_k),\tfrac12 h(u_k)\right),
\]
we deduce from Koebe's distortion theorem \eqref{la2a} and
\eqref{e} that
\begin{equation}\label{f}
B\left(w_k,\tfrac{4}{25} r_k \right) \subset W_k  \subset
B\left(w_k,2 r_k\right).
\end{equation}
Now let $\delta$ be a small positive number to be fixed later. We
put
\begin{equation}\label{f2}
m_k=\left[ \delta r_kx^p\right]
\quad \text{and} \quad
\rho_k=\frac{h'(u_k)}{x^p}.
\end{equation}
Note that if $0\leq l\leq m_k$ and $\delta<\frac13$, then, for
large $x$,
\begin{equation}\label{f0}
l\delta \rho_k +\delta^2 \rho_k \leq   m_k \delta \rho_k +
\delta^2 \rho_k \leq 2 \delta^2 r_kx^p  \rho_k =2 \delta^2  h(u_k)
< \tfrac14 h(u_k).
\end{equation}
For $0\leq l\leq m_k$ we now define
\[
v_{k,l}=\varphi_k\left(F(w_k)+l\delta \rho_k\right),
\]
\[
S_{k,l}=S\left(F(w_k)+l\delta \rho_k,\delta^2 \rho_k\right),
\]
\[
V_{k,l}=\varphi_k\left( S_{k,l} \right)
\]
and
\[
J_{k,l}=\left\{ u\in\R: \left| u-\left( h(u_k)+l\delta \rho_k
\right) \right|\leq \delta^2 \rho_k\right\}.
\]
The interval $J_{k,l}$ is thus the projection of $S_{k,l}$ onto
the real axis. If $\delta$ is sufficiently small, the intervals
$J_{k,l}$ are pairwise disjoint, and so the same holds for the
squares $S_{k,l}$. By~\eqref{f0} the squares  $S_{k,l}$ are
contained in $S\left(F(w_k),\frac14 h(u_k)\right)$ and thus
\begin{equation}\label{f3}
v_{k,l}\in V_{k,l}\subset W_k.
\end{equation}

We want to show that $S_{k,l}$ is admissible for at least one
half of the indices~$l$, provided $c_0$ is sufficiently large. In
order to do so, we first note that~\eqref{c3} yields
\[
\delta^2 \rho_k =\frac{\delta^2}{x^p} h'(u_k) \geq
\frac{\delta^2}{x^p} h'\left(\tfrac{1}{2} x\right) \geq
\frac{\delta^2}{x^p}  \exp\left(\tfrac{1}{28} x\right) >  100,
\]
if $c_0$ and hence $x$ is sufficiently large. Also, by \eqref{f0},
\[
\delta^2 \rho_k < \tfrac14  h(u_k) = \tfrac14\re F(w_k).
\]This means that each square $S_{k,l}$ has the size required to
be admissible. Denoting  by $I_k$ the set of all
$l\in\{0,1,\dots,m_k\}$ for which $S_{k,l}$ is admissible, we thus
have
\[
I_k=\{l: \length(J_{k,l} \cap L) \geq  \tfrac74\delta^2 \rho_k \}.
\]
With  $I_k'= \{0,1,\dots,m_k\} \setminus I_k$ we obtain
$\length(J_{k,l} \setminus L)>\frac14\delta^2 \rho_k$ for $k\in
I_k'$. Now suppose that $|I_k|<\frac12 (m_k+1)$ so that
$|I_k'|\geq \frac12 (m_k+1)$. This implies that
\begin{eqnarray*}
\length\left(\left[\tfrac34 h(u_k), \tfrac54 h(u_k) \right]
\setminus L \right)
&\geq&
\sum_{l=0}^{m_k} \length\left(J_{k,l}\setminus L\right)\\
&\geq&
\sum_{l\in I_k'} \length\left(J_{k,l}\setminus L\right)\\
&\geq& \left| I_k' \right| \tfrac14 \delta^2 \rho_k \\
&\geq& \tfrac18 (m_k+1)  \delta^2 \rho_k \\
&\geq& \tfrac18 \delta r_k x^p  \delta^2 \rho_k \\
&=&  \tfrac18 \delta^3 h(u_k).
\end{eqnarray*}
Since $L$ has density $1$,
this is a contradiction if $c_0$ and hence $u_k$
is sufficiently large.
Thus
\begin{equation}\label{f1}
\left| I_k \right| \geq \tfrac12 (m_k+1) \geq \tfrac12\delta r_k
x^p.
\end{equation}

As already mentioned, it follows from~\eqref{f0} that the squares
$S_{k,l}$ are contained in $S\left(F(w_k), \tfrac14 h(u_k)
\right)$ and thus, in particular,
\begin{equation}\label{g1}
F(w_k)+l\delta \rho_k\in S\left(F(w_k), \tfrac14 h(u_k) \right)
\subset B\left(F(w_k), \tfrac12 h(u_k) \right)
\end{equation}
for  $0\leq l\leq m_k$. Koebe's distortion theorem~\eqref{la2b}
now yields
\[
\frac{1}{81} \leq \frac{\left|\varphi_k'\left( F(w_k)+l\delta
\rho_k \right)\right|} {\left|\varphi_k'\left(
F(w_k)\right)\right|} \leq 81
\]
and hence, by Koebe's distortion theorem~\eqref{la2a},
\[
V_{k,l} \subset B\left(v_{k,l} , 4 \left|\varphi_k'\left( F(w_k)+
l\delta \rho_k \right)\right|\delta^2\rho_k  \right) \subset
B\left(v_{k,l} , 324 \left|\varphi_k'\left( F(w_k)
\right)\right|\delta^2\rho_k\right).
\]
Using~\eqref{d1} and the definition of $\rho_k$ in \eqref{f2} we
obtain
\[
V_{k,l} \subset B\left(v_{k,l} , \frac{324\delta^2}{x^p} \right).
\]
Similarly, it follows from Koebe's $\frac14$-theorem~\eqref{la2c}
that
\[
V_{k,l} \supset B\left(v_{k,l} , \frac{\delta^2}{324x^p} \right).
\]
With $c_1=\delta^2/324$ and $c_2=324 \delta^2$ we thus have
\begin{equation}\label{g}
B\left(v_{k,l} , \frac{c_1}{x^p} \right) \subset V_{k,l} \subset
B\left(v_{k,l} , \frac{c_2}{x^p}\right).
\end{equation}
Next we note that it follows from Lemma~\ref{la3} together with
\eqref{g1} and \eqref{f2} that
\begin{eqnarray*}
& &
\left| v_{k,l+1} - v_{k,l} -\frac{\delta}{x^p} \right|\\
&=&
\left| \varphi_k\left( F(w_k)+(l+1)\delta \rho_k \right)
-\varphi_k\left( F(w_k)+l\delta \rho_k \right)
- \varphi_k'\left( F(w_k)\right) \delta \rho_k \right|\\
&\leq&
26 \left|\varphi_k'\left( F(w_k)\right)\right|
\frac{(l+1)\delta\rho_k}{h(u_k)}\delta \rho_k\\
&=&
\frac{26 \delta^2 (l+1)\rho_k^2}{h'(u_k)h(u_k)}\\
&\leq&
 \frac{26 \delta^3 r_k x^p \rho_k^2}{h'(u_k)h(u_k)}\\
&=&
 \frac{26 \delta^3}{x^p}
\end{eqnarray*}
for $0\leq l\leq m_k-1$.
It follows that
\[
\re v_{k,l+1} -\re v_{k,l}  \geq \frac{\delta}{x^p} - \frac{26 \delta^3}{x^p}.
\]
For sufficiently small $\delta$ we have $\delta-26 \delta^3
>324 \tau \delta^2=\tau c_2$.
Hence
\begin{equation}\label{f4}
\re v_{k,l+1} -\re v_{k,l}  \geq \frac{\tau c_2}{x^p}.
\end{equation}
If $k,k'\in\{ 1,2,\dots,n\}$, $k\neq k'$, $l\in I_k$ and $l'\in
I_{k'}$, then $\re  v_{k,l}\in (u_k-2r_k,u_k+2r_k)$ by \eqref{f3}
and \eqref{f}. On the other hand, $\re  v_{k',l'}\notin (u_k-3
r_k,u_k+3 r_k)$ since, by \eqref{d3},
\[
(u_k-3 r_k,u_k+3 r_k)\cap (u_{k'}-3 r_{k'},u_{k'}+3 r_{k'})=\emptyset.
\]
Since $u_k \in L' \subset L$, it follows from \eqref{c} that
\[
\left| \re  v_{k,l} - \re   v_{k',l'}\right| \geq
r_k=\frac{h(u_k)}{h'(u_k)} \geq \frac{1}{u_k^p} \geq
\left(\frac45\right)^p \frac{1}{x^p}
\]
and thus
\begin{equation}\label{f5}
\left| \re  v_{k,l} - \re   v_{k',l'}\right| \geq \frac{\tau
c_2}{x^p}
\end{equation}
if $\delta$ and hence $c_2$ is sufficiently small.

We also note that it follows from Koebe's distortion
theorem~\eqref{la2a} together with \eqref{d1}, \eqref{f2},
\eqref{d5} and \eqref{c1} that
\begin{eqnarray*}
\left| v_{k,l} -w_k\right| &=& \left| \varphi_k\left(
F(w_k)+l\delta \rho_k \right)
- \varphi_k\left( F(w_k)\right)\right|\\
&\leq&
\left| \varphi_k'\left( F(w_k)\right)\right| 4 l \delta \rho_k\\
&\leq&
\frac{4 m_k\delta \rho_k}{h'(u_k)}\\
&\leq&
\frac{4 \delta^2 r_k x^p \rho_k}{h'(u_k)}\\
&=&
4 \delta^2 r_k \\
&\leq& 16\pi \delta^2.
\end{eqnarray*}
For small $\delta$ we thus have
\[
B\left(v_{k,l},\frac{\tau c_2}{x^p}\right) \subset B(w_k,1).
\]
Recall that $\re w_k = u_k \in L'$ and so it follows from
\eqref{d4} and \eqref{d2} that
\[
w_k\in  \left\{ \zeta \in\C: |\re \zeta -\re z|\leq \tfrac14 r-1,
|\im  \zeta -\im z|\leq \pi\right\}.
\]
Thus
\begin{equation}\label{B1}
B\left(v_{k,l},\frac{\tau c_2}{x^p}\right) \subset \left\{ \zeta
\in\C: |\re \zeta -\re z|\leq \tfrac14 r, |\im  \zeta -\im z|\leq
\pi+1 \right\}.
\end{equation}

Finally we note that it follows from \eqref{f3} that
\[
F(v_{k,l})\in F(V_{k,l}) \subset F(W_k) =  S\left(F(w_k),\tfrac14
h(u_k)\right). \]
 Also, since $u_k \in L'$, it follows from
\eqref{d4} that
\[
\re F(v_{k,l})\geq \re F(w_k)-\tfrac14 h(u_k)=\tfrac34 h(u_k)
\geq \tfrac34 h\left(x-\tfrac14 r\right)\geq \tfrac34
h\left(\tfrac78 x\right).
\]
Using \eqref{c2} we obtain
\begin{equation}\label{h1}
\re F(v_{k,l}) \geq \exp\left( \tfrac{1}{15}x\right),
\end{equation}
if $c_0$ and hence $x$ is sufficiently large.

We now put
\[
X=\left\{ v_{k,l}: k\in \{1,2,\dots,n\} , l\in I_k\right\}.
\]
Then
\[
m=|X| =\sum_{k=1}^n |I_k| \geq \tfrac12 \delta x^p \sum_{k=1}^n
r_k \geq \tfrac{1}{240}\delta x^p r
\]
by~\eqref{f1} and ~\eqref{d}. Thus $m\geq c_3  x^p r$ for
$c_3=\delta/240$. By \eqref{f4} and \eqref{f5} we can write
$X=\{a_1,a_2,\dots,a_m\}$ with $\re a_1<\re a_2<\dots <\re a_m$
and, putting $A_j=V_{k,l}$ if $a_j=v_{k,l}$, we deduce from
\eqref{g}, \eqref{f4}, \eqref{f5}, \eqref{B1} and \eqref{h1}
that, if $\delta$ is chosen to be sufficiently small, then
\eqref{A}, \eqref{B}, \eqref{C} and~\eqref{D} hold. Finally, it
follows from the construction that $F$ maps $A_j$ bijectively
onto an admissible square centred at $F(a_j)$, for
$j=1,2,\dots,m$.
\end{proof}

\section{Proof of Theorem~\ref{thm2}}

We now use Lemma~\ref{la5} to construct the set $E_p$. Let $c_0$
be the constant obtained from Lemma~\ref{la5} for fixed $\tau>1$.
(The condition for $\tau$ will be specified later.) Let $Q_0$ be
an admissible square $S(z_0,r_0)$ such that $\re z_0=x_0> c_0$.
For each $n\in\N\cup\{0\}$ we will define a finite collection
${\mathcal E}_n$ of compact, pairwise disjoint subsets of $Q_0$
with the following properties:
 for each $Q\in{\mathcal E}_n$, the set
$ F^n(Q)$ is an admissible square, each  $Q\in{\mathcal E}_n$
contains at least one element of ${\mathcal E}_{n+1}$ and each
$Q'\in{\mathcal E}_{n+1}$ is contained in a unique $Q\in{\mathcal
E}_n$.

 We start by putting ${\mathcal E}_{0}=\{Q_0\}$. Now suppose
that ${\mathcal E}_{n}$ has been defined and let $Q\in{\mathcal
E}_{n}$. Let $A_1,\ldots , A_m$ be the sets obtained by applying
Lemma~\ref{la5} to the admissible square $S(z,r)=F^n(Q)$. For
$k\in\Z$ and $j\in\{1,..,m\}$, we put
 $$ A_{j,k}=\{\zeta+2\pi ik: \zeta\in A_j\}.$$
 Now let $\varphi$ be the branch of the inverse function of $F^n$ that maps $S(z,r)$ to~$Q$. We define
 $${\mathcal E}_{n+1}(Q)=\left\{\varphi(A_{j,k}): A_{j,k}\subset S\left(z,\tfrac{1}{4}r\right)\right\}$$
and
$${\mathcal E}_{n+1}=\bigcup_{Q\in{\mathcal E}_n}{\mathcal E}_{n+1}(Q).$$
Then ${\mathcal E}_{n+1}$ has the required properties.

We define
$$\overline{{\mathcal E}}_n=\bigcup_{Q\in{\mathcal E}_n}Q$$
and
$$E_p=\bigcap_{n=0}^\infty \overline{{\mathcal E}}_n.$$
It follows from the construction and \eqref{C} that, for each
$z\in E_p$,
$$\lim_{n\to\infty}\re F^n(z)=\infty$$
as required. In fact it follows that $\exp(z)$ belongs to the set
$Z(f,D)$ defined in Section 1.

We estimate the Hausdorff dimension of $E_p$ using the following
result which is part of Frostman's Lemma; see, for
example,~\cite[Proposition 4.9]{Fal03}.
\begin{la} \label{la6}
Let $E$ be a compact subset of $\C$. Suppose that there exist a
probability measure $\mu$ supported on $E$ and positive constants
$c, \tilde{r}$ and $t$ such that, for each $z\in E$ and each $r
\in (0,\tilde{r})$,
$$\mu(B(z,r))\leq cr^t.$$
Then $\dim(E)\geq t$.
\end{la}

Following \cite{McM} we construct a sequence of probability
measures on $Q_0$. Let $\mu_0$ be the Lebesgue measure on $Q_0$
rescaled so that $\mu_0(Q_0)=1$. Then we construct the measure
$\mu_n$ supported on $\overline{{\mathcal E}}_n$ inductively.
Suppose that the measure $\mu_{n}$ on $\overline{{\mathcal
E}}_{n}$ has been defined and let $Q_{n} \in {\mathcal E}_{n}$.
The measure $\mu_{n+1}$ is defined as follows.  If $A \subset
\overline{{\mathcal E}}_{n+1}\cap Q_{n}$ then
\begin{equation}\label{measure}
\mu_{n+1}(A)=\frac{\area(Q_{n})}{\sum_{Q\in{\mathcal
E}_{n+1}(Q_{n})}\area(Q)}\mu_{n}(A)
\end{equation}
and, if $A \subset Q_0 \setminus \overline{{\mathcal E}}_{n+1}$,
then
$$\mu_{n+1}(A)=0.$$
Note that
$$\mu_{n+1}(\overline{{\mathcal E}}_{n+1}\cap Q_{n})=\mu_{n}(Q_{n})$$
and that, for every $k\geq n$,
$$\mu_k(Q_n)=\mu_n(Q_n).$$
Thus there exists a unique measure $\mu$ supported on $E_p$ such
that $$\mu(Q_n)=\mu_n(Q_n)$$ for each set $Q_n\in{\mathcal E}_n$
and each $n \in \N$.

We now let $z\in E_p$. Our aim is to estimate $\mu(B(z,r))$ for
$r$ sufficiently small. Let $Q_n(z)$ be the unique element of
${\mathcal E}_n$ that contains~$z$. Then, for each $n\in\N$, we
have $Q_n(z)\subset Q_{n-1}(z)$ and, by construction,
$F^n(Q_n(z))=S(z_n,r_n)$ for some admissible square $S(z_n,r_n)$.

Let $x_n=\re z_n$. Then, by \eqref{C},
 \begin{equation}\label{J}
x_n\geq \exp\left(\tfrac{1}{15}x_{n-1}\right).
\end{equation}
We put
$$d_n(z)=\diam Q_n(z)$$
and denote the density of $\overline{{\mathcal E}}_{n+1}$ in
$Q_n(z)$ by
$$
\Delta_n(z)=\frac{\sum_{Q\in{\mathcal E}_{n+1}(Q_n(z))}\area
Q}{\area Q_n(z)}.
$$
We now estimate the quantities $\Delta_n(z)$ and $d_n(z)$. In
order to do this, we first prove that there is a uniform bound
for the distortion of $F^n$ on each set $Q \in{\mathcal E}_n$.
(Recall that if a function $f$ is univalent on a set $S$ then the
{\it distortion} of $f$ on $S$ is $\sup_{u,v \in S}
\frac{|f'(u)|}{|f'(v)|}$.)

\begin{la} \label{la6a}
There exists $K>0$ such that, if $n \in \N$ and $Q \in{\mathcal
E}_n$ with $F^n(Q) = S(z',r')$, if $\varphi$ is the branch of
$F^{-n}$ that maps $S(z',r')$ to $Q$ and if $\tilde{Q} =
\varphi\left(B(z',\sqrt{2}r')\right)$, then
 \[
\sup_{u,v\in \tilde{Q}}\frac{|(F^n)'(u)|}{|(F^n)'(v)|}<K.
\]
\end{la}
\begin{proof}
Since the branch of $F^{-1}$ that maps $F^n(Q)$ to $F^{n-1}(Q)$
is univalent in $B(z', 2r')$, it follows from Koebe's distortion
theorem~\eqref{la2b} that the distortion of $F$ on
$F^{n-1}(\tilde{Q})$ is bounded by the constant
$$K_1=\frac{(\sqrt{2}+1)^4}{(\sqrt{2}-1)^4}.$$
Also, by construction, there is an admissible square $S(z'',r'')$
such that $F^{n-1}(Q)\subset S(z'',\tfrac{1}{4}r'')$ and
$F^{n-1}(\tilde{Q}) \subset B(z'',\tfrac{1}{2}r'')$. The branch of
$F^{-(n-1)}$ that maps $F^{n-1}(\tilde{Q})$ onto $\tilde{Q}$ is
univalent in $B(z'', r'')$ and so, by Koebe's distortion
theorem~\eqref{la2b}, the distortion of $F^{n-1}$ on $\tilde{Q}$
is bounded by the constant $K_2 = 81$. The result now follows by
putting $K=K_1K_2 < 10^4$.
\end{proof}

We now use the result of Lemma~\ref{la6a} to obtain estimates for
the density $\Delta_n(z)$ and the diameter $d_n(z)$.
\begin{la} \label{la7}
There exists a constant $c_5>0$ such that, for $n=0,1,2, \ldots$,
$$\Delta_n(z)\geq\frac{c_5}{x_n^p}.$$
\end{la}
\begin{proof}
It follows from Lemma~\ref{la6a} that
$$
\Delta_n(z)\geq\frac{1}{K^2}\frac{\sum_{Q\in{\mathcal
E}_{n+1}(Q_n(z))}\area F^n(Q)}{\area F^n(Q_n(z))}.$$ By
construction,
$$
F^n\left(\bigcup_{Q\in{\mathcal
E}_{n+1}(Q_n(z))}Q\right)=\bigcup_{A_{j,k}\subset
S(z_n,r_n/4)}A_{j,k},
$$
where $A_{j,k}=\{\zeta+2\pi ik:\zeta\in A_j\}$ and $A_j$ is one of
the sets obtained by applying Lemma~\ref{la5} to $F^n(Q_n(z)) =
S(z_n,r_n)$. Note that there are at least $c_3r_nx_n^p$ such sets
$A_j$ and, by \eqref{A}, each of these sets satisfies
$$\area A_j\geq\pi\frac {c_1^2}{x_n^{2p}}.$$
Also, for each $j$, the set $\left\{k:\ A_{j,k}\subset
S\left(z_n,\tfrac{r_n}{4}\right)\right\}$ has at least
$\frac{r_n}{4\pi}-2$ elements. Since $r_n>100$,
$$ \frac{r_n}{4\pi}-2>\frac{r_n}{8\pi}
$$
and so$$\Delta_n(z)\geq \frac{c_1^2c_3}{32K ^2}\frac{1}{x_n^p}.$$
\end{proof}

\begin{la} \label{la7a}
There exist constants $c_6,c_7>0$ such that, for $n = 0,1,2
\ldots$ and for each $ Q\in{\mathcal E}_{n+1}(Q_n(z))$,
 \begin{equation}\label{M1}
\frac{c_6}{|(F^n)'(z)|x_n^p}\leq\diam
Q\leq\frac{c_7}{|(F^n)'(z)|x_n^p}.
\end{equation}
In particular, for $Q=Q_{n+1}(z)$ we have
\begin{equation}\label{M}
\frac{c_6}{|(F^n)'(z)|x_n^p}\leq
d_{n+1}(z)\leq\frac{c_7}{|(F^n)'(z)|x_n^p}.
\end{equation}
\end{la}
\begin{proof}
Since $F^n(Q)$ is one of the sets $A_{j,k}$ in $S(z_n,r_n)$, it
follows from \eqref{A} that $ F^n(Q)$ contains a ball of radius
$c_1/x_n^p$ and is contained in a ball of radius $c_2/x_n^p$
which is contained in $S(z_n,r_n)$. Hence
$$\frac{2c_1}{x_n^p}\frac{1}{\sup_{u\in Q}|(F^n)'(u)|}\leq \diam Q\leq
\frac{2c_2}{x_n^p}\frac{1}{\inf_{u\in Q}|(F^n)'(u)|}.$$ It now
follows from Lemma~\ref{la6a} that
$$\frac{2c_1}{K}\frac{1}{x_n^p|(F^n)'(z)|}\leq \diam Q\leq
\frac{2c_2K}{x_n^p|(F^n)'(z)|}.$$
\end{proof}
We now obtain an estimate for the derivative $|(F^n)'(z)|$ in
terms of $x_n$.

\begin{la} \label{la8}
For each $\delta>0$, there exists $n_0>0$ such that, for $n>n_0$,
$$
|(F^n)'(z)|\geq \frac{x_n}{8 \pi}.$$
\end{la}
\begin{proof}
It follows from Koebe's $\tfrac{1}{4}$-theorem \eqref{la2c} that
if $\varphi$ is the branch of $F^{-1}$ that maps $F(z)$ to $z$
then
$$\varphi(B(F(z),\re F(z)))\supset B\left(z,\frac{\re F(z)}{4|(F'(z)|}\right).$$
Since $\varphi(B(F(z),\re F(z)))$ contains no vertical segments
of length $2\pi$ we obtain
$$|F'(z)|\geq\frac{\re F(z)}{4\pi}.$$
As $\re F^i(z)$ is much bigger than $4\pi$ for $i=1,\ldots ,n$ and
$\re F^n(z)\geq x_n-r_n\geq x_n/2$, it follows that
$$|(F^n)'(z)|\geq\frac{\re F^n(z)}{4\pi}\geq\frac{x_n}{8\pi}.$$
\end{proof}
  It follows from \eqref{M} and Lemma~\ref{la8}
 that, for large $n$,
\begin{equation}\label{r1}
d_{n+1}(z)\leq\frac{8\pi c_7}{x_n^{p+1}},
\end{equation}
 so $\lim_{n\to\infty}d_n(z)=0$ and
$$
\{z\}=\bigcap_{n=1}^\infty Q_n(z).$$ Since $Q_{n+1}(z) \subset
Q_n(z)$ we have $d_{n+1}(z) \leq d_n(z)$. Thus, for $r$
sufficiently small, there exists a unique $n$ such that
\begin{equation}\label{r}
d_{n+1}(z)\leq r<d_n(z).
\end{equation}

Now fix $\delta\in(0,1)$. We may assume that $r<1$ is small
enough to ensure that $n > n_0$, where $n_0$ is defined as in
Lemma~\ref{la8}. Before we estimate the measure $\mu$ of $B(z,r)$
we shall show that, for $\tau$ sufficiently large, the ball
$B(z,r)$ meets exactly one set in ${\mathcal E}_n$, namely the
set $Q_n(z)$. We now fix $\tau>2K+2$.
\begin{la} \label{la9}
For each $n\in\N$, if $Q,Q'\in{\mathcal E}_n(Q_{n-1}(z))$ then
$\dist (Q,Q')\geq \diam Q$.
\end{la}
\begin{proof}
Let $Q,Q'\in{\mathcal E}_n(Q_{n-1}(z))$. It follows from
Lemma~\ref{la6a} that
$$
\frac{\diam Q}{\dist(Q,Q')}\leq K\frac{\diam F^{n-1}(Q)}{\dist(F^{n-1}(Q),F^{n-1}(Q'))}.
$$
It follows from the construction, \eqref{A} and \eqref{D} that
$$
\diam F^{n-1}(Q)\leq\frac{2c_2}{x_{n-1}^p}\quad
\mbox{and}
\quad
\dist(F^{n-1}(Q),F^{n-1}(Q'))\geq\frac{(\tau-2)c_2}{x_{n-1}^p},
$$
and so
$$\frac{\diam Q}{\dist(Q,Q')}\leq\frac{2K}{\tau-2}.$$
The result now follows since $\tau > 2K +2$.
\end{proof}
Now let
$$
{\mathcal U}_n= \{Q\in {\mathcal E}_{n+1}(Q_n(z)):\ Q\cap
B(z,r)\neq\emptyset\}.
$$
Note that it follows from Lemma~\ref{la5}, Lemma~\ref{la6a}
and~\eqref{r} that, if $Q \in {\mathcal E}_n$ and \linebreak $Q
\cap B(z,r) \neq \emptyset$, then $Q = Q_n(z)$. So, by \eqref{r},
\eqref{measure}, Lemma~\ref{la7}, and \eqref{M1},
\begin{eqnarray*}
\mu(B(z,r))
&=&
\mu(B(z,r)\cap Q_n(z))\\
&\leq&
\sum_{Q\in{\mathcal U}_n}\mu(Q)\\
&=&
\sum_{Q\in{\mathcal U}_n}\mu_{n+1}(Q)\\
&\leq&
\sum_{Q\in{\mathcal U}_n}\left(\prod_{j=0}^n\Delta_j(z)\right)^{-1}\frac{\area Q}{\area Q_0}\\
&\leq& |{\mathcal U}_n|\ \frac{(x_0 \ldots
x_{n-1})^p}{x_n^p|(F^n)'(z)|^2}\frac{c_7^2}{c_5^{n+1}\area Q_0}.
\end{eqnarray*}
If $n$ is sufficiently large, then \[
 \mu(B(z,r)) \leq |{\mathcal U}_n|\
\frac{x_0^p(x_1\ldots x_{n-2})^{p+1}x_{n-1}^p}{x_n^p|(F^n)'(z)|^2}
\]
and so, by \eqref{J},
\begin{equation}\label{R}
\mu(B(z,r))\leq |{\mathcal U}_n|\
x_n^{-p}|(F^n)'(z)|^{-2}x_{n-1}^{p+\delta}.
\end{equation}
    In order to get an upper bound for $|{\mathcal U}_n|$ it is
    sufficient to estimate the number of sets $A_{j,k}$ in $S(z_n,r_n)$ which meet $F^n(B(z,r))$. By
    Lemma~\ref{la6a},
    \begin{equation}\label{R0}
    \diam \left( F^n(B(z,r))\cap S(z_n,r_n)\right)\leq 2K|(F^n)'(z)|r.
    \end{equation}
 It follows from \eqref{A} and \eqref{R0} that there are at most $K|(F^n)'(z)|rx_n^p/c_1$ values of $j$ for which
 $F^n(B(z,r)) \cap A_{j,k} \neq \emptyset$ for some $k\in \Z$. Also, for each
 such $j$, the maximum number of values of $k$ for which $F^n(B(z,r)) \cap A_{j,k} \neq
 \emptyset$ is at most
$\tfrac{K}{\pi}|(F^n)'(z)|r +1.$\\\\
Now we consider two cases.\\
{\it Case} 1:
\begin{equation} \label{R1a}
\ \frac{K}{\pi}|(F^n)'(z)|r<1. \end{equation}
 Then $$|{\mathcal
U}_n|\leq \frac{K}{c_1}|(F^n)'(z)|rx_n^p$$ and hence, by
\eqref{R},
\begin{equation}\label{R1}
\mu(B(z,r))\leq\frac{K}{c_1}r\frac{x_{n-1}^{p+\delta}}{|(F^n)'(z)|}.
\end{equation}
It follows from Lemma~\ref{la8} and \eqref{M} that
\begin{equation}\label{R2}
\frac{1}{|(F^n)'(z)|}\leq\left(\frac{(8\pi)^p}{x_n^p|(F^n)'(z)|}\right)^{1/(p+1)}\leq\left(\frac{(8\pi)^p}{c_6}\right)^{1/(p+1)}d_{n+1}(z)^{1/(p+1)}.
\end{equation}
By \eqref{R1a}, Lemma~\ref{la8} and \eqref{J},
\begin{equation}\label{R3}
r<\frac{\pi}{K |(F^n)'(z)|}\leq\frac{8\pi^2}{Kx_n} <
\frac{1}{x_{n-1}^{(p+\delta)/\delta}}.\end{equation} It follows
from  \eqref{R1}, \eqref{R2}, \eqref{R3} and \eqref{r} that there
exists a positive constant $c_8$ such that
$$
\mu(B(z,r))\leq c_8r^{1+\frac{1}{p+1}-\delta}.
$$
\medskip
{\it Case} 2:
\begin{equation}\label{R3a}
\frac{K}{\pi}|(F^n)'(z)|r\geq1. \end{equation}
 In this case, it follows from the discussion after \eqref{R0}
 that
$$ |{\mathcal U}_n|\leq\frac{2K^2}{\pi c_1}|(F^n)'(z)|^2r^2x_n^p,$$
so, by \eqref{R},
$$
\mu(B(z,r))\leq\frac{2K^2}{\pi c_1}x_{n-1}^{p+\delta}r^2.
$$
It follows from \eqref{r}, \eqref{M} and Lemma~\ref{la8} that
$$
r< d_n(z)\leq\frac{c_7}{x_{n-1}^p|(F^{n-1})'(z)|}\leq\frac{8\pi
c_7}{x_{n-1}^{p+1}}
$$
and hence
$$x_{n-1}^{p+\delta}\leq\left(\frac{8\pi c_7}{r}\right)^{\frac{p+\delta}{p+1}}.$$
Thus there exists a positive constant $c_9$ such that
$$\mu(B(z,r))\leq c_9 r^{1+\frac{1}{1+p}-\frac{\delta}{p+1}}.
$$
\medskip
In both cases, since $r<1$,  we have
$$\mu(B(z,r))\leq\max\{c_8, c_9\} r^{1+\frac{1}{1+p}-\delta},$$
so, by Lemma~\ref{la6}, for each $\delta\in(0,1)$,
$$\dim(E_p)\geq 1+\frac{1}{1+p}-\delta.$$
Letting $\delta$ tend to 0 we obtain that
$$\dim(E_p)\geq 1+\frac{1}{1+p}.$$
This completes the proof of Theorem~\ref{thm2}.

\section{Concluding remarks}

1. The Hausdorff dimension of the set $E_p$ constructed in the
proof of Theorem~\ref{thm2} is in fact equal to $1+1/(p+1)$. To
see this, consider the cover of $E_p$ by the sets in ${\mathcal
E}_n$. Let $Q_n \in {\mathcal E}_n$. Then there exists an
admissible square $S(z',r')$ with $r' < \tfrac{1}{2} x =
\tfrac{1}{2} \re z'$ and $F^n(Q_n) = S(z',r')$. It follows from
Lemma~\ref{la6a} and Lemma~\ref{la5} that, for each $s>1$,
\begin{eqnarray*}
\frac{\sum_{Q \in {\mathcal E}_{n+1}(Q_n)} (\diam Q)^s}{(\diam
Q_n)^s} & \leq & K^s\frac{\sum_{Q \in {\mathcal E}_{n+1}(Q_n)}
(\diam
F^n(Q))^s}{(\diam F^n(Q_n))^s}\\
& \leq & \frac{K^s}{(2\sqrt{2}r')^s}\frac{r'}{\pi}\frac{r'x^p}{c_1} \left(\frac{2c_2}{x^p}\right)^s\\
& \leq & c r'^{(2-s)}x^{p(1-s)}\\
& < & cx^{2-s+p(1-s)},
\end{eqnarray*}
where $c>0$ is a constant that is independent of $n \in \N$ and of
the choice of $Q_n \in {\mathcal E}_n$. Now suppose that $s = 1 +
\tfrac{1}{1+p} + \delta$, for some $\delta > 0$. Then
\[
 2 - s +p(1-s) = 1 - \delta - \frac{1}{1+p} - \frac{p}{1+p} -
 p\delta = - \delta(1+p).
\]
Thus, if $n$ and hence $x$ is sufficiently large,
\[
 \frac{\sum_{Q \in {\mathcal E}_{n+1}(Q_n)} (\diam Q)^s}{(\diam
Q_n)^s} < 1.
\]
Since $\max \{\diam Q: Q \in {\mathcal E}_n\} \to 0$ as $n \to
\infty$, it follows that $\dim E_p \leq s$. The result now follows
by letting $\delta \to 0$.

2. The examples in \cite{Sta00} of entire functions in the class
$B$ which show that the estimate in Theorem~\ref{thm1} is sharp
for $q>1$ have a logarithmic tract similar to the region
$$\Omega=\left\{ x+iy: x>1, y>\frac{x}{(\log x)^{q-1}}\right\}.$$
The region $\Omega$ also appears in \cite{KU} where it is shown
that, for $E_\lambda(z)=\lambda e^z$, the set of $z\in
I(E_\lambda)$ for which $E_\lambda^n(z)\in \Omega$ for large $n$
has Hausdorff dimension $1+1/q$.

3. Rempe~\cite{Rem} has recently shown that, if $f,g \in B$ and
there exist quasiconformal homeomorphisms $\phi,\psi:\C\to\C$
such that $\phi\circ f=g\circ \psi$, then there exists $R>0$ and
a quasiconformal homeomorphism $\theta:\C\to\C$ such that
$\theta(f(z))=g(\theta(z))$ if $|f^n(z)|\geq R$ for all $n\geq
0$. Since quasiconformal homeomorphisms map sets of Hausdorff
dimension~$2$ to  sets of Hausdorff dimension~$2$, this implies
that $\dim I(f)=2$ if $\dim I(g)=2$. Choosing $g=\lambda f$ with
sufficiently small $\lambda$ we see that, in order to prove that
$\dim I(f) = 2$ for all functions of finite order in the class
$B$, it is sufficient to consider such functions for which the
Fatou set consists of a single attracting basin.

Note that this kind of reasoning does not extend to the case
where the dimension is less than~$2$, since then the Hausdorff
dimension is not preserved by a quasiconformal homeomorphism. The
sharp bounds for the distortion of Hausdorff dimension under
quasiconformal mappings are given by a famous result of
Astala~\cite{Ast}.

In general, it is open as to whether two quasiconformally
equivalent functions $f$ and $g$ can have escaping sets of
different Hausdorff dimensions. It is known, however, that this
cannot happen when the maps $\phi$ and $\psi$ can be chosen to be
conformal.

\end{document}